\documentclass[12pt,nopagetitle]{article}
\usepackage{amsfonts, amsmath, amssymb, amsthm, enumitem}  
\usepackage[english]{babel}
\usepackage[utf8]{inputenc}
\usepackage[OT1]{fontenc}
\usepackage{bm}

\usepackage{textcomp, cmap, comment}	
\usepackage{graphicx, wrapfig}
\usepackage{epigraph, xcolor, hyperref}
\usepackage{array,tabularx,tabulary,booktabs}
\usepackage{euscript, mathrsfs}	

\newcounter{iii}

\newcommand{\bb}{{\mathcal B}}
\newcommand{\aaa}{{\mathcal A}}

\newcommand{\g}{\mathcal G}
\newcommand{\N}{\mathbb N}

\newcommand{\ff}{\mathcal F}
\theoremstyle{plain}
\newtheorem{thm}{Theorem}

\newtheorem{lem}{Lemma}
\newtheorem{cla}{Claim}
\newtheorem{prop}{Proposition}

\newtheorem{cor}{Corollary}
\newtheorem{obs}{Observation}

\theoremstyle{definition}
\newtheorem{defi}{Definition}
\newtheorem{conj}{Conjecture}

\title{Perfect matchings in down-sets}
\author{Peter Frankl\footnote{R\'enyi Institute, Budapest, Hungary and Moscow Institute of Physics and Technology, Russia; Email: {\tt peter.frankl@gmail.com}} \, and Andrey Kupavskii\footnote{G-SCOP, CNRS, University Grenoble-Alpes, France and
Moscow Institute of Physics and Technology, Russia; Email: {\tt kupavskii@ya.ru}. }}

\begin{document}
\maketitle
\begin{abstract}
  In this paper, we show that, given two down-sets (simplicial complexes) there is a matching between them that matches disjoint sets and covers the smaller of the two down-sets. This result generalizes an unpublished result of Berge from circa 1980. The result has nice corollaries for cross-intersecting families and Chv\'atal's conjecture. More concretely, we show that Chv\'atal's conjecture is true for intersecting families with covering number $2$.
  
A family $\mathcal F\subset 2^{[n]}$ is intersection-union (IU) if for any $A,B\in\mathcal F$ we have $1\le |A\cap B|\le n-1$. Using the aforementioned result, we derive several exact product- and sum-type results for IU-families.
\end{abstract}
\section{Introduction}

For a positive integer $n$ let $[n]:= \{1,\ldots, n\}$ be the standard $n$-element set. For or a set $X$ let $2^{X}$ denote its power set. Subsets of $2^X$ are called  {\it families}. A family $\ff\subset 2^{X}$ is {\it intersecting}  if for any $A,B \in \ff$ we have $A\cap B\ne \emptyset$ . Similarly, $\ff\subset 2^{X}$ is {\it union} if  for any $A,B \in \ff$ we have $A\cup B\ne X$. One should note that $\ff$ is union iff the family of complements, $\ff^c:=\{X\setminus A: A\in\ff\}$ is intersecting. 

Erd\H os, Ko and Rado were the first to investigate intersecting families. 
\begin{thm}[Non-uniform Erd\H os--Ko--Rado Theorem \cite{EKR}]\label{thmnew1}
suppose that $\ff\subset 2^{[n]}$ is intersecting. Then \begin{equation}\label{eqnew1}
  |\ff|\le 2^{n-1}.
\end{equation}
Moreover, there exists some intersecting family $\ff'\subset 2^{[n]}$ with $\ff\subset \ff'$ and $|\ff'| = 2^{n-1}$.
\end{thm} 
Note that the above result implies the existence of myriads of non-isomorphic intersecting families attaining equality in \eqref{eqnew1}.

Let us mention that \eqref{eqnew1} states that an intersecting family contains at most half of all subsets of $[n]$. This statement is nearly trivial. Namely, out of every pair $A,[n]\setminus A$ an intersecting family may contain at most one.

\begin{defi}
  A family $\bb\subset 2^X$ is called a \emph{down-set} (\emph{up-set}) if for all $B\in \bb$ and $A\subset X$ with $A\subset B$ ($B\subset A$), respectively, $A\in\bb$ holds. 
\end{defi}
There is an important {\it correlation inequality} concerning down-sets and up-sets.

\begin{thm}[Harris-Kleitman Inequality \cite{Har}, \cite{K}] Suppose that $\aaa\subset 2^{[n]}$ is an up-set and $\bb\subset 2^{[n]}$ is a down-set. then 
\begin{equation}\label{eqnew2}
  \frac{|\aaa\cap \bb|}{2^n}\le \frac{|\aaa|}{2^n}\frac{|\bb|}{2^n}.
\end{equation}
\end{thm}
 For a family $\ff\subset 2^{X}$ let $\ff^{\uparrow}$ and $\ff^{\downarrow}$ stand for the up-set and  down-set generated by $\ff$, respectively. That is,
 $$\ff^\uparrow:=\big\{G: \exists F\in\ff, F\subset G\subset X\big\}, \ \ \ \ff^\downarrow:=\big\{G: \exists F\in\ff, G\subset F\big\}.$$
 As one can see, $\ff^\downarrow$ is independent of $X$.
 
 The most natural examples of intersecting families are {\it stars}, that is, families for which there is an element common to all its members.
 \begin{conj}[Chv\'atal \cite{C}]
   Suppose that $\mathcal D\subset 2^X$ is a down-set, $\ff\subset\mathcal D$ is intersecting. then
   \begin{equation}\label{eqnew3}
     |\ff|\le \max_{x\in X}\big\{|F\in \mathcal D: x\in F\}|\big\}=:\Delta(\mathcal D).
   \end{equation}
 \end{conj}
 Unfortunately, after half a century, very little is known about Chv\'atal's conjecture. The Harris--Kleitman Correlation Inequality implies the following weaker result. 
 \begin{thm}
   Suppose that $\bb\subset 2^{[n]}$ is a down-set and $\ff\subset \bb$ is intersecting. Then
   \begin{equation}\label{eqnew4}
     |\ff|\le |\bb|/2.
   \end{equation}
 \end{thm}

 \begin{proof}
   Set $\aaa:=\ff^\uparrow$. Then $\ff\subset \aaa\cap \bb$, $\aaa$ is an intersecting up-set. In view of \eqref{eqnew1}, $|\aaa|\le 2^{n-1}$. Now \eqref{eqnew4} follows from \eqref{eqnew2}.
 \end{proof}  
 We refer to \cite{FKKK} for much more on Chv\'atal's conjecture and correlation inequalities. Berge \cite{B} gave a different proof of \eqref{eqnew4} based on basic graph theory. To a family $\ff$ correspond its Kneser graph $KG(\mathcal F)$ with vertex set $\ff$ and edge set $\{(F,G): F\cap G = \emptyset\}$.
 
 \begin{thm}[Berge \cite{B}]\label{thmnew4} Suppose that $\bb$ is a down-set. Then (i) or (ii) holds.
 \begin{itemize}
   \item[(i)] $|\bb|$ is even and $KG(\bb)$ contains a perfect matching;
   \item[(ii)] $|\bb|$ is odd and $KG(\bb\setminus\{\emptyset\})$ contains a perfect matching. 
 \end{itemize}
 \end{thm}
 It should be clear that Theorem~\ref{thmnew4} implies \eqref{eqnew4}. We will present one proof of this theorem at the end of Section~\ref{sec3}.
 
 One of our main results is a two-families version of Theorem~\ref{thmnew4}. To state it let us take two arbitrary families $\ff$ and $\g$ and define the bipartite Kneser graph $KG(\ff,\g)$ with partite sets $\ff$ and $\g$, with edge set $\{(F,G): F\in\ff, G\in \g, F\cap G = \emptyset\}$.

\begin{thm}\label{thm22}
  Suppose that $\ff$ and $\g$ are down-sets, $|\ff|\le |\g|$. Then there is a perfect matching of $\ff$ in $KG(\ff,\g)$.
\end{thm}

The key in proving this theorem is to enlarge the class of objects from families to monotone functions. We give the general statement and the proof in Section~\ref{sec2}.

One of the nice corollaries of this result is the following theorem.
\begin{thm}\label{crossiu}
  If $\ff,\g\subset 2^{[n]}$ are cross-intersecting, then \begin{equation}\label{eqcrossiu}|\g|+|\ff|\le \max\{|\ff^{\downarrow}|,|\g^{\downarrow}|\}.\end{equation}
\end{thm}

We provide a modest contribution to Chv\'atal's conjecture as well. 
For a family $\ff$ of non-empty subsets one defines the {\it covering number} $\tau(\ff)$ as the minimal integer $t$ such that there exists a $t$-set $T$ with $T\cap F\ne \emptyset$ for all $F\in \ff$. 
Let us note that stars have covering number $1$. In many problems involving intersecting families, families with covering number $2$ are suboptimal.

\begin{thm}\label{thmchv}
  Suppose that $\ff\subset \g\subset 2^{[n]}$, $\g$ is a down-set and $\ff$ is intersecting. If $\tau(\ff)\le 2$ then \eqref{eqnew3} holds.
\end{thm}

Our remaining results concern so-called IU-families.
\begin{defi}
  A family $\ff\subset 2^X$ is an intersecting-union family or IU-family for short if for all $F,G\in\ff$ both $F\cap G\ne \emptyset$ and $F\cup G\ne X$ hold. 
\end{defi}

{\bf Example.} Let $[n] = X\cup Y$ be a partition, $\aaa\subset 2^X$ an intersecting family and $B\subset 2^Y$ a union family. Define
$\aaa\times \bb:=\{A\cup B: A\in\aaa, B\in\bb\}.$
Then $|\aaa\times\bb| = |\aaa||\bb|$ and $\aaa\times \bb$ is an IU-family. 

Using Theorem~\ref{thmnew1} one can construct many IU-families of size $2^{n-2}$. Apparently unaware of \eqref{eqnew2} several people, including Daykin--Lov\'asz \cite{DL}, Sch\"onheim and Seymour (private communication, cf. \cite{DL}) proved that $2^{n-2}$ is the maximum.

\begin{thm}\label{thmnew7} If $\ff\subset 2^{[n]}$ is IU then 
\begin{equation}\label{eqiu}
  |\ff|\le 2^{n-2}.
\end{equation}
\end{thm}
\begin{proof}
  Since $\ff^\uparrow$ is intersecting and $\ff^\downarrow$ is union, $\ff\subset \ff^\uparrow\cap \ff^\downarrow$ implies \eqref{eqiu} via \eqref{eqnew2}:
  $$\frac{|\ff|}{2^n}\le \frac{|\ff^\uparrow|}{2^n}\frac{|\ff^\downarrow|}{2^n}\le \frac 12\cdot\frac 12 = \frac 14,$$
  yielding \eqref{eqiu}.
\end{proof}

Note that Theorem~\ref{crossiu} implies a stronger version of \eqref{eqiu}: if we assume in \eqref{eqcrossiu} that $\ff,\g$ are both union, then so are $\ff^\downarrow$ and $\g^\downarrow,$ and so the right hand side of \eqref{eqcrossiu} is at most $2^{n-1}$. Putting $\ff = \g$, we recover \eqref{eqiu}.

Actually, \eqref{eqnew2} implies a two-families version as well. Given two families $\ff,\g\subset 2^{[n]}$, we say that they are {\it cross-intersecting} if $A\cap B\ne \emptyset$ for any $A\in\ff,B\in\g$. Cross-union and cross-IU are defined analogously.

\begin{thm}\label{thm1new}
  Suppose that $\aaa,\bb\subset 2^{[n]}$ are cross-IU. Then
  \begin{equation}\label{eq1new}
    |\aaa||\bb|\le 2^{2n-4}.
  \end{equation}
\end{thm}

Applying \eqref{eq1new} to $\ff = \aaa$, $\ff = \bb$ obviously yields \eqref{eqiu}. Considering the sum $|\aaa|+|\bb|$ would only yield $|\aaa|+|\bb|\le 2^n$ because of the trivial choice $\aaa = 2^{[n]}, \bb = \emptyset$. To circumvent this problem we consider several families that are pairwise IU.

\begin{thm}\label{thm2new}
  Let $\aaa_1,\aaa_2,\ldots, \aaa_d\subset 2^{[n]}$ be pairwise cross-IU. Then
  $$\sum_{i=1}^d|\aaa_i|\le \max\big\{2^n, d\cdot 2^{n-2}\big\}.$$
\end{thm}
From the proof it is also clear that the equality holds only if $d\le 4$ and $\aaa_1 = 2^{[n]}$ for some $i$ or $d\ge 4$ and $\aaa_1 = \ldots = \aaa_d$. We should mention that Hilton proved an analogous result for pairwise intersecting families $\aaa_i\subset {[n]\choose k}$.

\newpage

\section{Proof of Theorem~\ref{thm22}}\label{sec2}

In what follows we assume that $\mathbb N$ includes $0$. As we have mentioned, the key to the proof is to work with monotone functions $f:=2^{[n]}\to \mathbb N$. 
We say that a function $f$ is  {\it monotone (decreasing)}, if for any set $A$ and  element $i$ we have $f(A)\le f(A\setminus \{i\})$. For a function $f:2^{[n]}\to \N$, put $|f| = \sum_{X\in 2^{[n]}}f(X)$. Similarly, for a function $p:2^{[n]}\times 2^{[n]}\to \N$, put $|p| = \sum_{X,Y\in 2^{[n]}}p(X,Y).$
\begin{thm}\label{crossberge} Given two monotone functions $f,g: 2^{[n]}\to \mathbb N$, $|f|\le |g|$, there is a function $p: 2^{[n]}\times 2^{[n]}\to \mathbb N$, such that \begin{itemize}
\item[(i)] $p(X,Y)\ne 0$ only if $X,Y\subset [n]$ are disjoint;
\item[(ii)] $|p| = |f|$.
\item[(iii)] We have $\sum_{X\subset 2^{[n]}}p(X,Y)\le g(Y)$ for every $Y\subset [n]$ and  $\sum_{Y\subset 2^{[n]}}p(X,Y)= f(X)$ for every $X\subset [n]$;
\end{itemize}
\end{thm}
One should think of the function $p$ as of a weighted matching between disjoint sets (condition (i)) that `respects' the restrictions on the number of occurrences of each set, imposed by $f,g$ (condition (iii)). Condition (ii) states that `$p$ is a matching of $f$'. (Condition (ii) is actually superfluous since it follows from the second part of (iii) via a summation over $X\subset [n]$.)


If we take $f,g$ to be the characteristic functions of the down-sets $\ff, \g$ (i.e., $f(X) = 1$ if $X\in \ff$ and $f(X) = 0$ if $X\notin \ff$), then the statement above is equivalent to Theorem~\ref{thm22}. Indeed, for each $A\in \ff$ we can put $\phi(A) = B$, where $B$ is the only set such that $p(A,B)\ne 0$. Condition (i) guarantees that $A\cap B = \emptyset$, condition (iii) guarantees that $\phi: \ff\to \g$ is injective and covers all $\ff$, i.e., it is a matching of $\ff$.
\begin{proof}
  To simplify the presentation, we assume that $|f| = |g|$. Otherwise, reduce some of the values of $g(F)$, $F\in 2^{[n]}$, so that $|g|=|f|$, while preserving monotonicity. Clearly, if we construct the desired function $p$ for such $g$, then the same $p$ will work for the original $g$. In this case, we have equality in the inequality from (iii) for any $X$.

  The proof is by induction on $n$. The statement is easy to see for $n = 1$. Assume that it is true for monotone functions on $[2,n]$ and let us prove it for $[n]$. For a pair of functions as in the statement, define two functions $f',g':2^{[2,n]}\to \mathbb N$ as follows: for each $X\subset [2,n]$, put $f'(X) = f(X)+f(X\cup\{1\})$ and $g'(X) = g(X)+g(X\cup\{1\})$.

  By induction, there is a function $p'$ for  $f',g'$ as in the statement of the theorem. Now form a bipartite multigraph $G$ between two copies $\aaa,\bb$ of $2^{[2,n]}$, where for any $X\in \aaa,Y\in\bb$ we have $p'(X,Y)$ edges between sets $X$ and $Y$. For a set $F \in 2^{[2,n]}$, we denote by $F_a,F_b$ its copies in $\aaa$ and $\bb$, respectively. By the condition (iii) applied to $p'$, the degree $d_{F_a}$ of $F_a$ in $G$ is $f'(F)$ and   $d_{F_b}=g'(F)$.

  We shall need the following simple claim.
  \begin{cla}\label{cla1} Let $G$ be a bipartite multigraph and denote $d_v$ the degree of $v\in V(G).$ Assume that each $v$ is assigned a number $u_v$ such that $2u_v\le d_v$. Then there exists a subset of edges $W\subset E(G)$ and an  orientation on the edges of $W$, such that the outdegree of $v$ is exactly $u_v$ for each $v$.
  \end{cla}
  \begin{proof}
    The proof is by induction on $|E|$. This is trivial for $|E|=0$. First, assume that there is a vertex $v$ of degree $1$ in the graph and $w$ is the only neighbor of $v$. Then $u_v = 0$. If $u_w = 0$ then simply delete the edge $uw$. If $u_w>0$ then orient the edge $vw$ towards $w$, reducing $u_w$ by $1$. Repeating this, we may assume that any vertex in $G$ has either degree $0$ or at least $2$. In particular, there is an (even) cycle in $G$.
    
    Next, take an even cycle $C = \{v_1,\ldots, v_k\}$ in $G$ (it may have length $2$ there are two edges between $v_1$ and $v_2$). If $u_{v_i}>0$ then orient the edge $v_iv_{i+1}$ (where indices are cyclic modulo $k$) from $v_i$ to $v_{i+1}$. Otherwise, leave this edge without an orientation. Include the oriented edges of $C$ into $W$, remove the edges of $C$ from $G$ and reduce every non-zero $u_{v_i}$, $i=1,\ldots, k$, by $1$. We obtain a new graph $G'$ and an assignment $u'_v$, $v\in V(G')=V(G)$. We apply induction to this graph and assignment, but first we want to make sure that the condition $2u'_v\le d'_v$ for every $v\in V(G')$, where $d'_v$ is the degree of $v$ in $G'$. Indeed, the degree only changed for the vertices of $C$, and, whenever $u_{v_i}$ was non-zero, $u'_{v_i} =u_{v_i}-1$ and $d'_{v_i} =d_{v_i}-2$.
  \end{proof}

  Next, apply Claim~\ref{cla1} to $G$ with $u_{F_a} = f(F\cup \{1\})$ and $u_{F_b} = g(F\cup\{1\})$. Note that this is a correct assignment of $u$'s since $2u_{F_a} = 2f(F\cup\{1\})\le f(F)+ f(F\cup \{1\}) = f'(F) = d_{F_a}(G)$, and similarly for $u_{F_b}$. We are ready to define the function $p$. For a pair of sets $X,Y\subset 2^{[2,n]}$
   \begin{itemize}
     \item we put $p(X\cup \{1\},Y)$ to be the number of edges in $G$ between $X_a,Y_b$ and that are oriented from $X_a$ to $Y_b$;
     \item we put $p(X, Y\cup \{1\})$ to be the number of edges in $G$ between $X_a,Y_b$ and that are oriented from $Y_b$ to $X_a$;
     \item we put $p(X,Y)$ to be the number of non-oriented edges in $G$ between $X_a$ and $Y_b$.
   \end{itemize}
Let us verify that $p$ has all the required properties.   First, $p$ satisfies condition (i) from the proposition (by the definition and since $p'$ satisfied it). Second, $|p|=|E(G)|=|p'|= |f'|= |f|=|g|$. Third, for any $Y\subset [2,n]$ we have $$\sum_{X\subset 2^{[n]}}p(X,Y\cup\{1\}) = \sum_{X\subset 2^{[2,n]}}p(X,Y\cup\{1\})=u_{Y_b} = g(Y\cup \{1\}),$$ $$\sum_{X\subset 2^{[n]}}p(X,Y) = \sum_{X\subset 2^{[2,n]}}p(X,Y)+p(X\cup\{1\},Y)=d_{Y_b}-u_{Y_b}=g'(Y)-g(Y\cup\{1\}) = g(Y).$$ The symmetric equalities (roles of $X$ and $Y$, as well as $f$ and $g$ being switched) are also valid and are checked analogously. This verifies (iii).
\end{proof}

\section{Other proofs}\label{sec3}
\begin{proof}[Proof of Theorem~\ref{crossiu}]
  W.l.o.g., assume that $|\ff^{\downarrow}|\le|\g^{\downarrow}|$. From Theorem~\ref{thm22} it follows that we can obtain a matching of  $\ff^{\downarrow}$ in $KG(\ff^\downarrow, \g^\downarrow)$. 
  Clearly, out of each pair at most one can be included in the respective family  ($\ff$ or $\g$), and thus $|\ff|+|\g|\le |\g^{\downarrow}|$.
\end{proof}

\begin{proof}[Proof of Theorem~\ref{thmchv}]
  Suppose that $\tau(\ff) = 2$ and by symmetry that $\{1,2\}$ is a cover. For $i\in[2]$ define
  $$\ff_i:=\{F\setminus \{i\}: F\in\ff, F\cap [2] = \{i\}\},$$
  $$\ff_{12}:=\{F\setminus\{1,2\}: F\in \ff, \{1,2\}\subset F\}. $$
We think of these families as subfamilies of $2^{[3,n]}$. Since $\{1,2\}$ is a cover for $\ff$,
$$|\ff| = |\ff_1|+|\ff_2|+|\ff_{12}|.$$
Note that, with the analogous notations $|\g(1)| = |\g_1|+|\g_{12}$, $|\g(2)| = |\g_2|+|\g_{12}|$ and obviously $\ff_{12}\subset \g_{12}$. Thus \eqref{eqnew3} will follow from
\begin{equation}\label{eq28}
  |\ff_1|+|\ff_2|\le \max\{|\g_1|,|\g_2|\}.
\end{equation}
 Since $\ff$ is intersecting, $\ff_1$ and $\ff_2$ are cross-intersecting. Noting that $\ff_1,\ff_2\subset 2^{[3,n]}$ are down-sets, Theorem~\ref{crossiu} implies that
 $$|\ff_1|+|\ff_2|\le \max\{|\ff_1^{\downarrow}|,|\ff_2^{\downarrow}|\}\le \max\{|\g_1|,|\g_2|\}.$$
\end{proof}

\begin{proof}[Proof of Theorem~\ref{thm1new}]
  Set $\alpha:=|\aaa|/2^n, \alpha^\uparrow:=|\aaa^\uparrow|/2^n, \alpha^\downarrow:=|\aaa^\downarrow|/2^n$ and define $\beta, \beta^\uparrow,\beta^\downarrow$ analogously. Note that the Harris--Kleitman correlation inequality implies
  \begin{equation}\label{eq2new}\alpha\le \alpha^\uparrow\alpha^\downarrow, \ \  \ \ \beta\le \beta^\uparrow\beta^\downarrow.\end{equation}
  Since $\aaa^\uparrow$ and $\bb^\uparrow$ are cross-intersecting, $\alpha^\uparrow+\beta^\uparrow\le1$, implying
  \begin{equation}\label{eq3new}
    \alpha^\uparrow\beta^\uparrow\le\frac 14.
  \end{equation}
  Similarly, $\aaa^\downarrow$ and $\bb^\downarrow$ are cross-union, yielding
  \begin{equation}\label{eq4new}
    \alpha^\downarrow\beta^\downarrow\le \frac 14.
  \end{equation}
  Combining \eqref{eq2new},\eqref{eq3new},\eqref{eq4new}, we immediately obtain
  $$\alpha\beta\le(\alpha^\uparrow\beta^\uparrow)(\alpha^\downarrow\beta^\downarrow)\le 2^{-4}.$$
\end{proof}
Theorem~\ref{thm2new} immediately follows from the next two results: Corollary~\ref{cor1new} and Theorem~\ref{thm3new}. But first we need a simple lemma.
\begin{lem}\label{lem1new}
  Suppose that $d\ge 1$ and $1\le x\le d$. Then \begin{equation}\label{eq5new}                                              x+\frac dx\le 1+d.
  \end{equation}
\end{lem}
\begin{proof}
  Rearranging yields $(x-1)(x-d)\le 0$, which is true for $1\le x\le d$.
\end{proof}

\begin{cor}\label{cor1new}
  Suppose that $\aaa_1,\ldots, \aaa_d\subset 2^{[n]}$ are pairwise cross-IU families, $d\ge 5$ and $|\aaa_1|\ge |\aaa_2|\ge\ldots \ge|\aaa_d|.$ Then
  \begin{equation}\label{eq6new}
    |\aaa_1|+\ldots+|\aaa_d|\le d2^{n-2}.
  \end{equation}
  Moreover, the inequality is strict unless $\aaa_1 = \ldots = \aaa_d$. 
\end{cor}
\begin{proof}
 Set $a_i:=|\aaa_i|/2^{n-2}$. We have $a_1\ge a_1\ge\ldots\ge a_d$. In proving \eqref{eq6new} we may assume $1\le a_1\le 4$. In view of \eqref{eq1new}, $a_i\le 1/a_1$ for $2\le i\le d$. Moreover, if there is $F\in \aaa_i\setminus \aaa_j$ for some $i,j\ge 2$, then $\aaa_1$ and $\{F\}\cup \aaa_j$ are cross-IU, and we get that $a_j<1/a_1$. Similarly, if there is a set $F\in \aaa_i\Delta\aaa_1$, then $\{F\}\cup \aaa_1$ and $\aaa_j$ are cross IU for any $j\ne 1,i$, and again we get $a_j<1/a_1$. Now \eqref{eq5new} implies $a_1+\ldots+a_d\le a_1+\frac {d-1}{a_1}\le d$, as desired, and the first inequality is strict unless $\aaa_1 \supset \aaa_2 =\ldots = \aaa_d$. The second inequality is strict unless $a_1 = \ldots  =a_d$, in which case the previous sentence implies $\aaa_1 =\ldots = \aaa_d$, or if $a_1 = d-1$. This is only possible for $d=5$, but then $\aaa_2 = \ldots = \aaa_d = \emptyset$, and the inequality \eqref{eq6new} is strict.
\end{proof}

\begin{thm}\label{thm3new}
  Let $\aaa,\bb\subset 2^{[n]}$ be cross-IU families, $|\aaa|\ge |\bb|$. then
  \begin{equation}\label{eq7new}
   |\aaa|+3|\bb|\le 2^n.
  \end{equation}
\end{thm}
\begin{proof}
  Set $x:=|\aaa|/2^n, y = |\bb|/2^n$. By \eqref{eq1new}, we know that $xy\le 1/16$. If $x\le 3/4$ then \eqref{eq7new} follows from \eqref{eq5new}. For convenience, set $z = 1-x$, and note that w.l.o.g. $0\le z\le \frac 14$.

  Note that $(1-q)q$ is a decreasing function of $q$ for $0\le q\le \frac 12$. Obviously, $|\aaa^\uparrow|\ge |\aaa|$, implying $|\aaa^\uparrow||\bb^\uparrow|\le 2^{2n}z(1-z)$ and the same for $|\aaa^\downarrow||\bb^\downarrow|$. Using the Kleitman-Harris correlation inequality, this leads to
  $$\frac{|\aaa||\bb|}{2^{2n}}\le z^2(1-z)^2$$
  and thereby to $|\bb|/2^n\le z^2(1-z) = z(z(1-z))<\frac z4$. consequently,
  $$\frac {|\aaa|}{2^n}+3\frac {|\bb|}{2^n}\le 1-z+\frac {3z}4 = 1-\frac z4\le 1.$$
\end{proof}

In the remainder of this section, we will present a proof of Theorem~\ref{thmnew4}. 

Let us say that a family $\ff$ can be {\it matched to itself} if either (i) or (ii) from Theorem~\ref{thmnew4} holds $\ff$, i.e., if either $\ff$ or $\ff\setminus \emptyset$ can be partitioned into pairs of disjoint sets. Let us restate Theorem~\ref{thmnew4} for convenience.

\begin{thm}\label{berge}
  If $\aaa\subset 2^{[n]}$ is a down-set, then it can be matched to itself.
\end{thm}

We use the following simple observation.
\begin{obs}\label{lem3}
  Let $\g_1,\g_2\subset {X\choose 2}$ be two matchings, then $\g_1\cup\g_2$ is bipartite.
\end{obs}

\begin{proof}[Proof of theorem~\ref{berge}]
  The proof is by induction on $n$. Let $\aaa\subset 2^{[n]}$ be a down-set. Consider the two families $\aaa(n),\aaa(\bar n)\subset 2^{[n-1]}$. They are both down-sets, and $\aaa(n)\subset \aaa(\bar n)$.

  Let $(A_i,B_i)$, $1\le i\le |\aaa(n)|/2$ be a perfect matching of $\aaa(n)$ or $\aaa(n)\setminus \emptyset$. Denote this matching $\g_1$.   Let $(C_i,D_i)$, $1\le i\le |\aaa(\bar n)|/2$ be a perfect matching of $\aaa(\bar n)$ or $\aaa(\bar n)\setminus \emptyset$. The edges with both $C_i$ and $D_i$ form the matching $\g_2$. By Observation~\ref{lem3}, the graph $\g_1\cup \g_2$ is bipartite. Consequently, we can reorder some pairs (replacing $(A_i,B_i)$ with $(B_i,A_i)$ when necessary) to make $\mathcal I:=\{A_i: 1\le i\le |\aaa(n)|/2\}$ an independent set in $\g_1\cup \g_2$.

  Next, we define the matching $(A_i, B_i\cup\{n\})$, $1\le i\le |\aaa(n)|/2$ and $(C_j',D_j')$, where $C_j' \in\{C_j,C_j\cup \{n\}\}$, $D_j' \in \{D_j,D_j\cup \{n\}\}$, and we add $\{n\}$ only if the corresponding set belongs to $\mathcal I$. Since $\mathcal I$ is an independent set in $\g_2$, at most one set from $C_j,D_j$ is enlarged. Thus, we get pairs of disjoint sets that cover almost all $\aaa$.

  In case when both $|\aaa(n)|$ and $|\aaa(\bar n)|$ are odd, the only two unmatched sets in $\aaa$ are $\{n\}$ and $\emptyset$, so we add an extra pair $(\{n\},\emptyset)$ to the matching. In case when only $|\aaa(\bar n)|$ is odd, then the only unmatched set in $\aaa$ is $\emptyset$, and so the matching constructed a paragraph earlier is already the matching of $\aaa$. In case when only $\aaa(n)$ is odd, then replace $\aaa(\bar n)$ with $\aaa'(\bar n)$ in the previous argument, where $\aaa'(\bar n) = \aaa(\bar n)\setminus A$, and $A\in \aaa(\bar n)\setminus \aaa(n)$ is any inclusion-maximal set. Then both $\aaa(n)$ and $\aaa'(n)$ are odd, and the argument above gives a matching that covers all sets of $\aaa$, except for $\{n\}, \emptyset$ and $A$. Since $A\in \aaa(\bar n)$, we can then add a pair $(\{n\},A)$ to the matching, obtaining the matching of $\aaa$.
\end{proof}

\section{Concluding remarks and open problems}
In the present paper we proved results related to IU-families. It is very natural to consider the following quantitative version.

\begin{defi}
  Let $t$ and $s$ be positive integers and $\ff\subset 2^{[n]}$. If $|A\cap B|\ge t$ for any $A,B\in \ff$ then $\ff$ is called \emph{ $t$-intersecting}. If $|A\cup B|\le n-s$ for any $A,B\in \ff$ then $\ff$ is called \emph{$s$-union}. Finally, if $\ff$ satisfies both properties, then we call it an \emph{$(t,s)$-family.}
\end{defi}
Let $m(n,t) = \max\{|\ff|:\ff\subset 2^{[n]} \text{ is } t-\text{intersecting}\}.$ The exact value of $m(n,t)$ was determined by Katona \cite{Ka1}. Let $m(n,t,s)$ denote the maximum of $|\ff|$ over all $(t,s)$-families $\ff\subset 2^{[n]}$. 

Attaching an extra element to the ground set without any sets containing it shows that 
\begin{equation}\label{eq5.3}
  m(n+1,t,1)\ge m(n,t).
\end{equation}

Katona \cite{Ka2} conjectured that equality holds in \eqref{eq5.3}. This was proved in \cite{F75}. Extending the example above in the natural way shows
\begin{equation}\label{eq5.4}
  m(n+n',t,s)\ge m(n,t)m(n',s).
\end{equation}
\begin{conj} Suppose that $n\ge t+s$. Then 
\begin{equation}\label{eq5.5}
  m(n,t,s) = \max\big\{m(n',t)m(n-n',s): t\le n'\le n-s\big\}.
\end{equation}
\end{conj}
Defining the {\emph cross-$(t,s)$} property in the obvious way we have 
\begin{conj} Let $n\ge t+s$ and suppose that $\ff,\g\subset 2^{[n]}$ are cross-$(t,s)$ families. Then
\begin{equation}\label{eq5.6}
  |\ff||\g|\le m(n,t,s)^2.
\end{equation}
\end{conj}
Let us note that \eqref{eq5.6} was proved in \cite[Chapter 13]{FT} for the special case $s = 1$.

\end{document}